\documentclass[11pt]{article}
\usepackage[T1]{fontenc}
\usepackage[utf8]{inputenc}
\usepackage{lmodern}
\usepackage[margin=1in]{geometry}
\usepackage{amsmath,amssymb,amsthm}
\usepackage{booktabs,array}
\usepackage{microtype}
\usepackage[hidelinks]{hyperref}
\hypersetup{pdftitle={Nullstellensatz degree under Hajos joins and vertex identifications},
 pdfsubject={Graph coloring and algebraic proof complexity},
 pdfcreator={},pdfproducer={}}
\ifdefined\pdfinfoomitdate\pdfinfoomitdate=1\fi
\ifdefined\pdfsuppressptexinfo\pdfsuppressptexinfo=15\fi
\ifdefined\pdftrailerid\pdftrailerid{}\fi
\newtheorem{theorem}{Theorem}[section]
\newtheorem{lemma}[theorem]{Lemma}
\newtheorem{proposition}[theorem]{Proposition}
\newtheorem{corollary}[theorem]{Corollary}
\theoremstyle{remark}
\newtheorem{remark}[theorem]{Remark}
\newcommand{\F}{\mathbb F}

\newcommand{\N}{N_{k,\F}}
\newcommand{\HJ}{Haj\'os}

\newcommand{\tw}{\operatorname{tw}}

\title{Nullstellensatz degree under \HJ\ joins\\
and vertex identifications}
\author{Ying Xie\\Kennesaw State University\\\texttt{yxie2@kennesaw.edu}}
\hypersetup{pdfauthor={Ying Xie}}

\date{September 13, 2026}

\begin{document}
\maketitle
\begin{abstract}
We study the minimum coefficient degree $N_{k,\F}(G)$ of a Nullstellensatz
certificate for Bayer's $k$-coloring equations, where the characteristic
of $\F$ does not divide $k$. If $J$ is a \HJ\ join of non-$k$-colorable
graphs $G,H$ and $m=\max\{N_{k,\F}(G),N_{k,\F}(H)\}$, then
$N_{k,\F}(J)\leq m+k$. When deletion of the selected edge makes each
input $k$-colorable, we also have $N_{k,\F}(J)\geq m$; the degree
congruence then gives $N_{k,\F}(J)\in\{m,m+k\}$. This partially answers
a question of Li, Lowenstein, and Omar. For three-coloring over $\F_2$,
we construct an infinite $4$-critical family of exact degree seven,
attaining the bound at input degree four. In contrast, every graph
constructed from $K_4$ solely by \HJ\ joins has degree $O(\log n)$
and a certificate with polynomially many terms: joins preserve
treewidth at most three, and balanced separators yield low-degree
certificates. Additional vertex identifications are excluded from
this obstruction. We classify all single identifications of the
$25$-vertex base graph; exactly $36$ preserve degree seven, producing
$24$-vertex $4$-critical graphs of treewidth four. A compressed self-join
at adjacent true twins prevents degree loss and gives a repeatable
rule adding four vertices per round. The rule does not establish
degree amplification or preservation of criticality. Exact witnesses
and standalone verification programs accompany the finite results.
\end{abstract}

\section{Introduction}

Bayer's encoding expresses graph colorability as a system of polynomial
equations. For a non-colorable graph, a Nullstellensatz certificate is a
polynomial identity proving that these equations have no common zero.
At a fixed coefficient degree, finding a certificate reduces to solving
a linear system; see De Loera et al.~\cite{DLMRS} and
Romero and Tun\c{c}el~\cite{RT}.

Li, Lowenstein, and Omar~\cite[Problem~3.5]{LLO} ask how the minimum
Nullstellensatz degree of a graph obtained by a \HJ\ join depends on
the degrees of the two input graphs. The question appears for general
$k$ in~\cite[Problem~6.3]{RT}. We prove that, for critical inputs, the
output degree is either $m$ or $m+k$, where $m$ is the larger input degree.

The upper bound comes from a three-variable identity that combines the
input certificates. For the lower bound, a coloring of one input with
its selected edge removed gives a substitution by scalar multiples of
a single variable. This substitution preserves degree; descent from a
splitting field gives the comparison over the original field.

Section~\ref{sec:degree7} gives an explicit family attaining the bound
with both input degrees equal to four. The inputs are closed diamond
chains; their joins have exact degree seven over $\F_2$. The lower bound
combines a checked dual certificate for a $25$-vertex graph with a
substitution that folds longer chains onto the base construction.

Section~\ref{sec:joins-only} proves a complementary obstruction:
arbitrary joins from bounded-treewidth inputs retain bounded treewidth,
and a separator argument gives logarithmic certificate degree.
The result covers unbalanced join spines as well as balanced trees.
Section~\ref{sec:identifications} then studies degree-preserving
identifications of the base graph and a compressed self-join. These
operations control graph order without proving repeated degree rises.

The encoding and degree congruence are standard, and the small examples
in Section~\ref{sec:examples} belong to the classification in~\cite{LLO}.
Lauria and Nordstr\"om~\cite{LN} give explicit graph families requiring
linear degree, with exponential-size lower bounds for Boolean
polynomial calculus. Their degree lower bounds also apply to Bayer's
encoding. Our degree-seven family and its identified quotients have
fixed degree. The asymptotic question here is whether specified joins
and identifications can construct linear-degree graphs while their
order grows linearly; no such family is obtained in this paper.

\section{Encoding and degree}
\label{sec:prelim}

All graphs are finite, simple, and undirected. Fix $k\geq2$ and a field
$\F$ with $k\cdot1_{\F}\ne0$. For a graph $G=(V,E)$, set
\begin{equation}\label{eq:encoding}
 p_v=x_v^k-1,\qquad
 q_{uv}=q(x_u,x_v),\qquad
 q(s,t)=\sum_{j=0}^{k-1}s^{k-1-j}t^j.
\end{equation}
The common zeros over an algebraic closure of $\F$ are exactly the proper
$k$-colorings with colors chosen among the $k$th roots of unity. Indeed,
$q(s,t)=0$ for distinct $k$th roots, whereas $q(s,s)=ks^{k-1}\ne0$.

For non-$k$-colorable $G$, let $\N(G)$ be the least integer $d$ for which
there is an identity in the ordinary polynomial ring
\begin{equation}\label{eq:ordinary}
 1=\sum_{v\in V}r_vp_v+\sum_{uv\in E}h_{uv}q_{uv},
 \qquad \deg r_v,\deg h_{uv}\leq d.
\end{equation}
This is \emph{coefficient degree}, as in~\cite{LLO,RT}, rather than the
maximum degree of a product $r_vp_v$ or $h_{uv}q_{uv}$.
We write $N(G)$ when $k$ and $\F$ are fixed.

Let
\[
 R_V=\F[x_v:v\in V]/\langle x_v^k-1:v\in V\rangle.
\]
Every element has a unique representative in which each exponent is
less than $k$. Its degree means the total degree of that representative.
Reduction to this form never increases total degree. The ring $R_V$
uses only the vertex relations; edge polynomials are not set to zero.

\begin{lemma}[Quotient degree]\label{lem:quotient}
The minimum in~\eqref{eq:ordinary} equals the minimum degree of edge
multipliers in an identity
\begin{equation}\label{eq:quotient}
 1=\sum_{uv\in E}h_{uv}q_{uv}\qquad\text{in }R_V.
\end{equation}
If the edge multipliers in~\eqref{eq:quotient} have degree at most $d$,
the corresponding ordinary identity can be chosen with vertex multipliers
of degree at most $d-1$.
\end{lemma}
\begin{proof}
Reducing~\eqref{eq:ordinary} gives~\eqref{eq:quotient} without increasing
edge-multiplier degrees. Conversely, take reduced multipliers in
\eqref{eq:quotient}. The ordinary polynomial
$S=\sum h_{uv}q_{uv}$ has degree at most $d+k-1$ and reduces to $1$.
Each monomial reduction $M\mapsto M/x_v^k$ contributes
$(M/x_v^k)(x_v^k-1)$, whose multiplier has degree at most $d-1$.
Telescoping these reductions and changing signs gives
\eqref{eq:ordinary}. A degree-zero quotient certificate cannot exist:
constant multiples of the $q_{uv}$ have degree $k-1$, undergo no
reductions, and have zero constant coefficient. Thus the argument only
needs $d\geq1$.
\end{proof}

\begin{lemma}[Degree congruence, {\cite[Theorem~2.1]{DLMRS}}]
\label{lem:congruence}
For every non-$k$-colorable graph $G$, $N(G)\equiv1\pmod k$.
\end{lemma}
\begin{proof}
Put $d=N(G)$ and choose a quotient certificate of coefficient degree $d$,
using Lemma~\ref{lem:quotient}. For each multiplier $h_{uv}$, let
$h_{uv}^{[1]}$ be the sum of its reduced monomials whose total degrees
are congruent to $1$ modulo $k$.

The reduced monomial basis gives $R_V$ a direct-sum grading by total
degree modulo $k$: the relations $x_v^k=1$ preserve this grading.
Since each $q_{uv}$ has grade $k-1$, the grade-zero component of the
certificate is
\[
 1=\sum_{uv\in E}h_{uv}^{[1]}q_{uv}.
\]
At least one $h_{uv}^{[1]}$ is nonzero. Let $d'$ be the maximum degree
among these nonzero multipliers. Then $d'\leq d$ and $d'\equiv1\pmod k$.
The displayed identity is a certificate of degree $d'$, so minimality
of $d$ and Lemma~\ref{lem:quotient} give $d\leq d'$. Hence $d=d'$.
\end{proof}

\begin{lemma}[Extension of coefficients]\label{lem:extension}
For every field extension $L/\F$ and every degree bound $d$, a
degree-at-most-$d$ certificate exists over $L$ if and only if one exists
over $\F$.
\end{lemma}
\begin{proof}
The conditions on the coefficients are a finite linear system over $\F$.
Its consistency is unchanged by extending the field. Alternatively,
apply an $\F$-linear map $L\to\F$ fixing $1$ coefficientwise to an
identity over $L$. The constraint polynomials have coefficients in
$\F$, so the resulting identity remains valid and has the same degree
bound.
\end{proof}

\section{The join bound}

Take vertex-disjoint graphs $G,H$ and oriented edges $(a,b)$ of $G$
and $(x,y)$ of $H$. Their \HJ\ join identifies $a$ with $x$, deletes
the selected edges, and adds the edge between $b$ and $y$. After the
identification, write $c$ for the image of $y$. Thus the deleted edges
are $ab,ac$ and the new edge is $bc$. The notation $J$ always includes
the choice of these oriented edges.

\begin{theorem}\label{thm:main}
Let $J$ be a \HJ\ join of non-$k$-colorable graphs $G,H$. Set
\[
 m=\max\{N(G),N(H)\}.
\]
Then $J$ is non-$k$-colorable and
\begin{equation}\label{eq:upper}
 N(J)\leq m+k.
\end{equation}
If $H-xy$ is $k$-colorable, then $N(J)\geq N(G)$. If $G-ab$ is
$k$-colorable, then $N(J)\geq N(H)$. In particular, if both deletions
are $k$-colorable, then
\begin{equation}\label{eq:twovalues}
 m\leq N(J)\leq m+k,\qquad N(J)\in\{m,m+k\}.
\end{equation}
\end{theorem}

A graph is $(k+1)$-critical if it has chromatic number $k+1$ and every
proper subgraph is $k$-colorable. Such inputs satisfy both deletion
hypotheses of Theorem~\ref{thm:main}.

\subsection{A local composition identity}

In the following calculation $a,b,c$ also denote the corresponding
variables. Define homogeneous polynomials of degree $k-2$ by
\begin{align}
 P(a,b,c)&=\sum_{j=0}^{k-2}c^j
             \sum_{r=0}^{k-2-j}b^{k-2-j-r}a^r,\label{eq:P}\\
 Q(a,c)&=\sum_{j=1}^{k-1}a^{k-1-j}
             \sum_{r=0}^{j-1}c^{j-1-r}a^r.\label{eq:Q}
\end{align}
Telescoping first in $b$ and then in $c$ gives the ordinary identity
\begin{equation}\label{eq:divided}
 q(b,c)-ka^{k-1}=(b-a)P(a,b,c)+(c-a)Q(a,c).
\end{equation}
Set
\begin{equation}\label{eq:T}
 T_G=-\frac ak(b-a)P(a,b,c),\qquad
 T_H=-\frac ak(c-a)Q(a,c).
\end{equation}
These polynomials have degree at most $k$. In the quotient by
$a^k=b^k=c^k=1$, they satisfy
\begin{equation}\label{eq:local}
 1=\frac akq(b,c)+T_G+T_H,
 \qquad T_Gq(a,b)=T_Hq(a,c)=0.
\end{equation}
The first equality follows from~\eqref{eq:divided}. The second uses
$(b-a)q(a,b)=b^k-a^k$ and its analogue for $a,c$.

\begin{proof}[Proof of the upper bound in Theorem~\ref{thm:main}]
Embed quotient certificates of the two inputs into the common ring
$R_{V(J)}$. They take the form
\begin{equation}\label{eq:inputs}
 1=Aq(a,b)+U,\qquad 1=Bq(a,c)+V,
\end{equation}
where $U$ and $V$ are combinations of retained edges of $G$ and $H$
respectively, with every edge multiplier of degree at most $m$.
Equations~\eqref{eq:local} and~\eqref{eq:inputs} imply
$T_G=T_GU$ and $T_H=T_HV$. Hence
\begin{equation}\label{eq:composition}
 \boxed{1=\frac akq(b,c)+T_GU+T_HV.}
\end{equation}
All edge constraints on the right belong to $J$. Their multipliers
have degree at most $m+k$, and the new-edge multiplier has degree $1$.
Lemma~\ref{lem:quotient} yields the ordinary certificate, proving
\eqref{eq:upper} and non-$k$-colorability.
\end{proof}

\subsection{A degree-preserving substitution}

\begin{proof}[Proof of the lower bounds in Theorem~\ref{thm:main}]
Assume $H-xy$ is $k$-colorable. Its endpoints $x,y$ receive the same
color in every proper $k$-coloring; otherwise that coloring would extend
to $H$. Let $L/\F$ be a finite splitting extension of $z^k-1$.
Label the colors of a chosen coloring by distinct $k$th roots
$\lambda_v\in L$, normalizing so that $\lambda_x=\lambda_y=1$.

Keep all variables from $G$ and make the substitution
\begin{equation}\label{eq:substitution}
 x_v\longmapsto\lambda_v x_a\quad(v\in V(H)).
\end{equation}
It agrees at the identified vertex, sends $x_c$ to $x_a$, and induces
a homomorphism between the quotient rings over $L$, since
$(\lambda_vx_a)^k-1=x_a^k-1$. Each retained edge of $G$ keeps its
constraint. Each retained edge $uv$ of $H$ maps to zero, because
\[
 q(\lambda_u x_a,\lambda_v x_a)
   =x_a^{k-1}q(\lambda_u,\lambda_v)=0
 \quad\text{when }\lambda_u\ne\lambda_v.
\]
The new-edge constraint $q(x_b,x_c)$ maps to $q(x_b,x_a)$, the
constraint of the deleted edge of $G$.

Consequently a degree-$d$ certificate for $J$ gives a certificate for
$G$ over $L$ whose degree is at most $d$. Substituting scalar multiples
of single variables and then reducing exponents cannot increase degree.
Lemma~\ref{lem:extension} descends the certificate to $\F$, so
$N(G)\leq d$. Interchanging the inputs proves the other lower bound.
If both hypotheses hold, Lemma~\ref{lem:congruence} restricts the
integer interval $[m,m+k]$ to $m$ and $m+k$.
\end{proof}

\begin{remark}
The substitution is a retraction on the variables of the retained input,
after extending coefficients. The criticality assumption is used only
to guarantee a coloring after deletion of the selected edge. No such
assumption is needed for the upper bound.
\end{remark}

\section{Examples over \texorpdfstring{$\F_2$}{F2}}
\label{sec:examples}

In this section $k=3$ and $\F=\F_2$. The local factors simplify to
\begin{equation}\label{eq:binary}
 T_G=1+ab^2+a^2c+abc,\qquad T_H=a^2c+ac^2,
\end{equation}
where $a^3=1$ has been used in the first expression. Thus
\eqref{eq:composition} is a particularly small explicit transformation.

Label $K_4$ by $0,1,2,3$. Join two copies along the oriented edges
$(0,1)$, identifying their vertex $0$ and labeling the other vertices
of the right copy by $4,5,6$ in order. The resulting Moser spindle $M$
has vertex set $\{0,\ldots,6\}$ and edge set
\begin{equation}\label{eq:M}
 E(M)=\{02,03,05,06,12,13,14,23,45,46,56\}.
\end{equation}
Here $ij$ denotes the unordered edge $\{i,j\}$. Let $C$ be the join
of $M$ along $(0,2)$ with $K_4$ along $(0,1)$, giving the other
vertices of the right copy the labels $7,8,9$. Thus
\begin{equation}\label{eq:C}
 E(C)=\{03,05,06,08,09,12,13,14,23,27,45,46,56,78,79,89\}.
\end{equation}

\begin{proposition}\label{prop:sharp}
The graphs $K_4,M,C$ are $4$-critical and have degrees
$N(K_4)=1$, $N(M)=4$, and $N(C)=4$. In particular, both alternatives
in~\eqref{eq:twovalues} occur, and the additive constant $3$ in the
upper bound cannot be decreased for $k=3$ over $\F_2$.
\end{proposition}
\begin{proof}
The degree of $K_4$ is witnessed by the linear multipliers in
Appendix~\ref{app:witnesses}; degree zero is impossible.
Theorem~\ref{thm:main} gives $N(M)\leq4$. The dual functional in the
appendix excludes degree $1$, and Lemma~\ref{lem:congruence} then gives
$N(M)=4$. For $C$, the same appendix gives a degree-$4$ certificate,
while the lower-bound part of Theorem~\ref{thm:main} gives
$N(C)\geq N(M)=4$.

For completeness, a join of $4$-critical graphs is $4$-critical.
Non-$3$-colorability follows from the join construction. If a retained
edge of the left input is removed, color that input with the edge
removed; its selected-edge endpoints have different colors. Color the
right input with its selected edge removed, so its selected endpoints
have the same color. Permute colors to agree at the identified vertex.
The added edge is then properly colored. The other side is symmetric.
If the added edge is removed, combine colorings of the two inputs with
their selected edges removed. Thus every edge deletion is $3$-colorable.
There are no isolated vertices, so all vertex deletions are also
$3$-colorable. An edge-deletion coloring can be changed to a $4$-coloring
by giving one endpoint a fourth color. Starting with $K_4$ proves the
criticality assertions for $M$ and $C$.
\end{proof}

These benchmark degrees occur in the classification of $4$-critical
graphs through twelve vertices in~\cite{LLO}. They show that both
endpoints of~\eqref{eq:twovalues} occur for $k=3$ over $\F_2$.

\section{An explicit family of exact degree seven}
\label{sec:degree7}

Throughout this section, $k=3$ and the coefficient field is $\F_2$.
All polynomial identities are in the vertex quotient $x_v^3=1$.
The proof uses a computed dual certificate for one graph and an algebraic
substitution to obtain the lower bound for the family.

\subsection{Construction and statement}

For $r\geq1$, define the closed diamond chain $D_r$ with terminals
$t_0,\ldots,t_r$. For each $0\leq i<r$, introduce two internal vertices
$p_i,q_i$ and the five edges
\[
 t_ip_i,\quad t_iq_i,\quad t_{i+1}p_i,\quad t_{i+1}q_i,\quad p_iq_i.
\]
Add the closing edge $t_0t_r$. Thus $D_r$ has $3r+1$ vertices and
$5r+1$ edges. For $r,s\geq2$, let $T_{r,s}$ be the \HJ\ join of disjoint
copies of $D_r,D_s$ along the oriented edge $(p_1,q_1)$ in each copy.
In particular, the two $p_1$ vertices are identified and the two $q_1$
vertices are connected by the new bridge.
The selected edge exists as soon as the chain has two diamonds.
The restriction $r,s\geq4$ below is used for the exact-degree argument,
which folds the chains onto $D_4$ and the join onto $T_{4,4}$.

\begin{theorem}[Exact degree-seven family]\label{thm:seven}
The graph $D_r$ is $4$-critical for every $r\geq1$, and $T_{r,s}$ is
$4$-critical for every $r,s\geq2$. For $r,s\geq4$,
\[
 N(D_r)=4,\qquad \boxed{N(T_{r,s})=7.}
\]
For every $r,s\geq2$, the graph $T_{r,s}$ has $3(r+s)+1$ vertices and
$5(r+s)+1$ edges. For $r,s\geq4$, it admits a degree-seven certificate
with at most $24(r+s)-3$ monomial-edge terms.
The base graph $T_{4,4}$ has $25$ vertices and $41$ edges, and the supplied
upper certificate has $189$ terms. Consequently, the upper bound of
Theorem~\ref{thm:main} is attained with both input degrees equal to four.
\end{theorem}

\begin{proof}[Criticality]
This argument applies to $D_r$ for $r\geq1$ and to $T_{r,s}$ for
$r,s\geq2$.
In a properly three-colored diamond, the two internal vertices have distinct
colors, forcing its two terminals to share the third color. A coloring of
$D_r$ would therefore make all terminals equal, violating the closing edge.

Deleting the closing edge permits a coloring with all terminals equal.
Deleting any edge of one diamond permits that diamond's terminals to have
distinct colors. For deletion of its internal edge, give both internal
vertices the third color. For deletion of a terminal-to-internal edge, give
that internal endpoint the color of the now nonadjacent terminal, and the
other internal vertex the third color. Assign one color to all terminals
before this diamond and a second color to all terminals after it. Every
intact diamond and the closing edge can then be properly colored.
Thus every edge deletion of $D_r$ is three-colorable. There are no isolated
vertices, so this also implies vertex-deletion colorability, and an
edge-deletion coloring extends to a four-coloring by recoloring one endpoint
with a fourth color. Hence $D_r$ is $4$-critical. The preservation argument
in Proposition~\ref{prop:sharp} proves criticality of $T_{r,s}$.
\end{proof}

\subsection{Explicit upper certificates}

For a diamond with terminals $u,v$ and internal vertices $p,z$, the following
five multipliers give a local equality identity:
\begin{equation}\label{eq:diamond-equality}
 u+v=\sum_{e\in\{up,uz,vp,vz,pz\}} f_e q_e,
\end{equation}
where vertex names also denote their variables and
\begin{center}
\begin{tabular}{@{}cl@{}}
\toprule
$e$ & $f_e$ \\
\midrule
$up$ & $uv+uz+v^2+vz$ \\
$uz$ & $v^2$ \\
$vp$ & $u^2+uv+uz+vz$ \\
$vz$ & $u^2$ \\
$pz$ & $u^2+v^2$ \\
\bottomrule
\end{tabular}
\end{center}
Expanding in the ordinary polynomial ring over $\F_2$ and cancelling
equal terms gives
\begin{align*}
 \sum_e f_e q_e
 &=u^3v+u^3z+uv^3+v^3z\\
 &=u^3(v+z)+v^3(u+z).
\end{align*}
Modulo $u^3=v^3=1$, this is $(v+z)+(u+z)=u+v$, proving
\eqref{eq:diamond-equality}.
It has twelve monomial-edge terms, all with coefficient degree two.
Summing it over the diamonds gives a coefficient-degree-two representation
of $t_0+t_r$. The quotient identity
\begin{equation}\label{eq:chain-upper}
 1=t_0 q_{t_0t_r}+t_0t_r(t_0+t_r)
\end{equation}
therefore gives a degree-at-most-four certificate for $D_r$ with at most
$12r+1$ terms. At the selected edge $p_1q_1$, its coefficient has exactly
the two terms obtained from $t_0t_r(t_1^2+t_2^2)$ when $r\geq4$.

For this family we use a compact alternative to~\eqref{eq:binary}.
Given input certificates $1=Aq_{ab}+U$ and $1=Bq_{ac}+V$ as
in~\eqref{eq:inputs}, one has
\begin{equation}\label{eq:binary-short}
 \boxed{1=bq_{bc}+bc(b+a)U+bc(a+c)V.}
\end{equation}
Indeed, $(b+a)q_{ab}=(a+c)q_{ac}=0$ in the quotient. Substituting
$U=1+Aq_{ab}$ and $V=1+Bq_{ac}$ reduces the right side to
$bq_{bc}+bc(b+c)=b^3=1$.
The two prefactors have degree three, proving $N(T_{r,s})\leq7$.
Removing the selected edge from the chain certificates leaves at most
$12r-1$ and $12s-1$ terms. Each prefactor has two monomials, so the resulting
certificate has at most $2(12r-1)+2(12s-1)+1=24(r+s)-3$ terms.
The file \texttt{anc/degree7/family.py} implements these formulas directly.

\subsection{The finite lower certificate}
\label{sec:seven-lower}

Label $t_i$ by $i$, $p_i$ by $r+1+2i$, and $q_i$ by
$r+2+2i$ in $D_r$. In $T_{4,4}$, retain the left labels $0,\ldots,12$,
identify right vertex $7$ with left vertex $7$, and label the remaining
right vertices $13,\ldots,24$ in increasing order of their old labels.
The resulting edge list is given in
\texttt{anc/degree7/base\_graph.json}.

\begin{lemma}[Finite certificate check]\label{lem:finite-seven}
There is an $\F_2$-linear functional $\Lambda$ on the reduced monomial
basis of $R_{V(T_{4,4})}$ such that
\begin{equation}\label{eq:dual-seven}
 \Lambda(1)=1,\qquad \Lambda(Mq_e)=0
 \quad(e\in E(T_{4,4}),\ \deg M\leq4).
\end{equation}
It vanishes outside total degrees zero, three, and six.
\end{lemma}
\begin{proof}[Computer-assisted verification]
The ancillary file \texttt{base\_lower.json} lists the reduced monomials
on which $\Lambda$ takes value one; it is zero on every unlisted monomial.
There are $104{,}833$ listed monomials, distributed as follows:
\begin{center}
\begin{tabular}{@{}rrrr@{}}
\toprule
Degree & $0$ & $3$ & $6$ \\
\midrule
Support size & $1$ & $883$ & $103{,}949$ \\
\bottomrule
\end{tabular}
\end{center}
Two exact arithmetic implementations validate the support, the value on
the constant, and every equality in~\eqref{eq:dual-seven}. The numbers of
reduced multiplier monomials of degrees zero through four are, respectively,
$1,25,325,2{,}900,19{,}850$. Hence the check evaluates
\[
 41(1+25+325+2{,}900+19{,}850)=947{,}141
\]
monomial-edge multiples. One implementation
uses tuples of variable indices and reduction by multiplicity modulo three.
The other uses base-three integer codes for exponent vectors and recursive
enumeration; it also compares the counts with the coefficients of
$(1+t+t^2)^{25}$. The second checker imports neither the search program nor
the first checker. Reproduction instructions and artifact hashes appear
in Appendix~\ref{app:reproduction}.
\end{proof}

Since $q_e$ has grade two modulo three, a multiplier monomial of degree
five or six gives only monomials of grade one or two after reduction.
The functional in Lemma~\ref{lem:finite-seven} vanishes on these grades.
It therefore annihilates every permitted product through coefficient
degree six. Applying it to a putative refutation would give $1=0$.
Thus $N(T_{4,4})\geq7$, and~\eqref{eq:binary-short} proves equality.

The supplied file \texttt{D4\_lower\_degree1.json} similarly gives a
degree-one obstruction for $D_4$. The checker verifies all $21(1+13)=294$
monomial-edge multiples. Lemma~\ref{lem:congruence} and
\eqref{eq:chain-upper} then prove $N(D_4)=4$.

\subsection{Folding over a field extension}
\label{sec:folding}

Let $\zeta^2+\zeta+1=0$ and $L=\F_4=\F_2(\zeta)$, so $\zeta^3=1$.
For every $r\geq4$ define a substitution from the vertex quotient of
$D_r$ to that of $D_4$ over $L$. On terminals set
\begin{equation}\label{eq:fold-terminals}
 t_0\mapsto t_0,\quad t_1\mapsto t_1,\quad t_2\mapsto t_2,
 \quad t_i\mapsto t_3\ (3\leq i<r),\quad t_r\mapsto t_4.
\end{equation}
Keep the internal vertices of diamonds $0,1,2$ at their counterparts in
$D_4$; send those of diamond $r-1$ to $p_3,q_3$. For each intervening
diamond $3\leq i\leq r-2$, set
\begin{equation}\label{eq:fold-internals}
 p_i\mapsto\zeta t_3,\qquad q_i\mapsto\zeta^2 t_3.
\end{equation}
There are no intervening diamonds when $r=4$.

Every variable image has cube one. All five edge polynomials of a collapsed
diamond map to zero: their two endpoints have distinct scalar factors from
$\{1,\zeta,\zeta^2\}$ multiplying $t_3$. Every other edge polynomial maps
to the corresponding edge polynomial of $D_4$, including the closing edge
and the selected edge $p_1q_1$. Each variable maps to a scalar times a single
variable, so substitution and reduction cannot increase coefficient degree.
By Lemma~\ref{lem:extension}, any certificate for $D_r$ therefore gives one
of no larger degree for $D_4$ over $\F_2$. It follows that $N(D_r)\geq4$.

Apply the same substitutions on both inputs of $T_{r,s}$. They agree at
the identified $p_1$ vertex and preserve the new bridge, inducing a map to
the vertex quotient of $T_{4,4}$ over $L$. Every retained edge maps to an
edge constraint or zero. A degree-$d$ refutation of $T_{r,s}$ would give
one of $T_{4,4}$ over $L$ of degree at most $d$. Applying the $\F_2$-linear
projection $a+b\zeta\mapsto a$ coefficientwise descends this refutation
to $\F_2$. Hence
\[
 N(T_{r,s})\geq N(T_{4,4})=7.
\]
This proves the lower bound for every $r,s\geq4$ and completes the proof
of Theorem~\ref{thm:seven}.

As finite checks of the formulas, the ancillary verifier constructs upper
certificates and checks every edge image for $(r,s)=(4,4),(5,4),(4,5),(6,7)$,
and $(16,16)$. The last graph has $97$ vertices and $161$ edges; its
certificate has $765$ terms.

\section{Construction depth}
\label{sec:depth}

\begin{corollary}\label{cor:depth}
Suppose $G$ has a construction from non-$k$-colorable leaf graphs of
degree at most $d_0$, using \HJ\ joins, identifications of nonadjacent
vertices, and additions of edges or vertices. Let $h$ be the maximum
number of join nodes on a leaf-to-root path. Then $N(G)\leq d_0+kh$.
\end{corollary}
\begin{proof}
A join increases the larger available bound by at most $k$.
Identifying nonadjacent vertices substitutes one variable for another,
sending each edge constraint to an edge constraint without increasing
degree. Adding edges or vertices preserves an existing certificate.
Induction through the construction proves the result.
\end{proof}

Any lower bound $N(G)\geq D$ therefore implies
$h\geq\lceil(D-d_0)/k\rceil$. This bounds construction depth. An
unbalanced tree can have large depth and few nodes, so a size lower
bound requires a separate argument.

\section{A limitation of joins alone}
\label{sec:joins-only}

The construction depth in Corollary~\ref{cor:depth} need not be
logarithmic in the graph order. Nevertheless, when additional vertex
identifications are excluded, a structural restriction gives a
logarithmic bound on certificate degree.

A tree decomposition of a graph consists of vertex sets, called bags,
indexed by a tree, such that every vertex occurs in a bag, every edge
has its endpoints in a common bag, and the bags containing any fixed
vertex induce a connected subtree. Its width is the largest bag size
minus one. The treewidth $\tw(G)$ is the minimum such width.

\begin{proposition}\label{prop:join-width}
For every \HJ\ join $J$ of $G$ and $H$,
\[
 \tw(J)\leq\max\{\tw(G),\tw(H),2\}.
\]
\end{proposition}
\begin{proof}
Use the notation of Theorem~\ref{thm:main}, with deleted edges $ab,ac$
and added edge $bc$. Take tree decompositions of the two inputs,
relabeling the shared vertex as $a$. Choose bags containing the
selected edges before deletion. Add a bag $\{a,b,c\}$ and connect it
to these two bags. The resulting bag graph is a tree. Every retained
edge remains covered, and the new bag covers $bc$. The bags containing
$a$ become connected through the new bag; the same condition holds
for $b,c$ and is unchanged for all other vertices. The maximum bag
size is at most the larger input bag size or three.
\end{proof}

\begin{theorem}[Separator certificate bound]\label{thm:separator}
Let $G$ be non-$k$-colorable, with $n\geq1$ vertices and treewidth at
most $w$. Set $Q=(w+1)(\lfloor\log_2 n\rfloor+1)$. Then
\begin{equation}\label{eq:separator-bound}
 N(G)\leq(k-1)Q.
\end{equation}
There is a quotient certificate with at most $k^{2Q}$ nonzero
monomial-edge terms. In particular, for fixed $k,w$, its degree is
$O(\log n)$ and its number of terms is polynomial in $n$.
\end{theorem}
\begin{proof}
Every induced subgraph on $m$ vertices has a separator of size at most
$w+1$ whose removal leaves components of order at most $m/2$. To see
this, restrict a width-$w$ decomposition to that subgraph and assign
each vertex's unit weight to one bag containing it. A weighted centroid
bag separates the decomposition tree into branches of weight at most
$m/2$. A vertex outside that bag has all of its bags in one branch.
Adjacent vertices outside the bag belong to the same branch. Each
remaining graph component therefore has at most $m/2$ vertices.

Consider a decision tree that, given any assignment of $k$ colors to
the vertices of $G$, finds a monochromatic edge by querying vertex
colors. Maintain a queried set $B$ with fixed colors and an active
unqueried set $U$ such that the coloring of $B$ has no proper extension
to $G[B\cup U]$. Initially $B=\varnothing$ and $U=V(G)$. Query a
balanced separator $S$ of $G[U]$. If the colors on $B\cup S$ already
exhibit a monochromatic edge, stop. Otherwise some component $C$ of
$G[U-S]$ has no extension compatible with those colors: if all
components extended, their colorings would combine to color
$G[B\cup U]$. Continue with $B\cup S$ and $C$.

The component is chosen from the graph and already queried colors,
without inspecting any unqueried color of the supplied assignment.
This defines a decision tree. The active order at least halves at
each stage, so every path uses at most $\lfloor\log_2 n\rfloor+1$
stages and at most $Q$ queries. Once no active vertices remain, the
invariant forces a monochromatic edge among the queried vertices.

Work in a splitting field $L$ of $x^k-1$, and write $\mu_k$ for its
$k$ distinct roots. For $c\in\mu_k$, the polynomial
\[
 I_c(x)=\frac1k\sum_{j=0}^{k-1}(x/c)^j
\]
is the indicator of $x=c$ on $\mu_k$. At each leaf $\ell$, multiply
the indicators for the queried vertex-color pairs to obtain $P_\ell$.
The leaf cylinders partition $\mu_k^n$, so $\sum_\ell P_\ell=1$.
If the edge returned at that leaf is $e_\ell$ and its common color
is $c_\ell$, then
\[
 P_\ell q_{e_\ell}=k c_\ell^{k-1}P_\ell.
\]
These are identities in $R_V\otimes_{\F}L$: evaluation on $\mu_k^n$
is injective on reduced polynomials, since the vertex polynomials
split into distinct roots. Consequently
\begin{equation}\label{eq:leaf-certificate}
 1=\sum_\ell\frac{P_\ell}{k c_\ell^{k-1}}q_{e_\ell}.
\end{equation}
Each multiplier has degree at most $(k-1)Q$ and at most $k^Q$ terms;
the tree has at most $k^Q$ leaves. Apply an $\F$-linear map
$L\to\F$ fixing $1$ coefficientwise, as in
Lemma~\ref{lem:extension}. This preserves the identity without
increasing degree or monomial support.
\end{proof}

The support bound also gives polynomial support in an ordinary
certificate. In the product of a reduced monomial and $q_{uv}$, only
the exponents of $x_u,x_v$ can need reduction, each at most once.
The telescoping argument of Lemma~\ref{lem:quotient} therefore adds
at most $2kS$ monomial-vertex terms to a quotient certificate with
$S$ monomial-edge terms.

\begin{corollary}[Joins alone]\label{cor:joins-only}
Every graph constructed solely by \HJ\ joins from graphs of
treewidth at most $w_0$ has treewidth at most $\max\{w_0,2\}$.
For fixed $k,w_0$, every non-$k$-colorable graph in this class has
coefficient degree $O(\log n)$ and a certificate with polynomially
many terms. In particular, for $k=3$ over $\F_2$, every graph
constructed solely by \HJ\ joins from $K_4$ satisfies
\begin{equation}\label{eq:k4-logarithmic}
 N(G)\leq8(\lfloor\log_2 n\rfloor+1).
\end{equation}
\end{corollary}
\begin{proof}
Apply Proposition~\ref{prop:join-width} inductively, followed by
Theorem~\ref{thm:separator}. Joins of copies of $K_4$ are
non-$3$-colorable by Theorem~\ref{thm:main}.
\end{proof}

Here ``solely'' excludes additional vertex identifications and edge
additions between joins; the identification within a \HJ\ join is
part of that operation. The selected edges and the construction shape
are arbitrary. Reuse of earlier graphs is permitted, with two disjoint
copies made for each join. Thus the corollary is not restricted to
iteration along the newly created bridge.

Binary branching alone does not yield~\eqref{eq:k4-logarithmic}.
A binary spine with $L$ leaves can have height $L-1$, while its
$K_4$-join graph has $n=3L+1$ vertices. Corollary~\ref{cor:depth}
alone then permits a linear degree bound. The separator argument
supplies the stronger restriction even for these unbalanced spines.
If a joins-only sequence had $N(G_t)=1+3t$, it would require
\[
 |V(G_t)|\geq2^{(1+3t)/8-1}.
\]
In particular, joins alone cannot give degree $\Omega(n^\epsilon)$
for any fixed $\epsilon>0$ as $n$ tends to infinity. This does not
assert a constant degree bound, nor does it prohibit unbounded degree
on graphs of exponentially growing order.

\section{Controlled vertex identifications}
\label{sec:identifications}

Identifying nonadjacent vertices cannot increase certificate degree,
but it can increase treewidth. We first classify which single
identifications preserve the degree of $T_{4,4}$, then give an operation
that prevents degree loss after a self-join. All finite results in this
section use $k=3$ and $\F=\F_2$.

\subsection{A complete classification for the base graph}

Use the labeling of $T=T_{4,4}$ in Section~\ref{sec:seven-lower}.
The untouched diamonds in each input chain have indices $0,2,3$.
Their internal vertices form the sets
\[
 L=\{5,6,9,10,11,12\},\qquad
 R=\{18,19,21,22,23,24\}.
\]
For nonadjacent $u,v$, write $T/(u=v)$ for the graph obtained by
identifying them and suppressing duplicate edges.

\begin{theorem}[Single-identification classification]\label{thm:identification}
Among the $259$ unordered nonadjacent vertex pairs of $T$, the quotient
$T/(u=v)$ has degree seven if and only if one endpoint lies in $L$
and the other in $R$. Each of these $36$ quotients has $24$ vertices
and $41$ edges, is $4$-critical, and has treewidth exactly four.
The full degree distribution is
\begin{center}
\begin{tabular}{@{}rrr@{}}
\toprule
Degree & Vertex pairs & Input-pair automorphism orbits \\
\midrule
$1$ & $20$ & $8$ \\
$4$ & $203$ & $64$ \\
$7$ & $36$ & $6$ \\
\bottomrule
\end{tabular}
\end{center}
\end{theorem}
\begin{proof}[Computer-assisted verification]
The ancillary data partition all nonedges of $T$ into $78$ classes,
with explicit automorphisms mapping each representative pair to every
member of its class. A standalone verifier reconstructs $T$ from its
definition, checks coverage of all $259$ pairs, and verifies each
permutation against the complete edge set. An exhaustive bijection
search, pruned only by necessary adjacency and vertex-color-refinement
invariants, enumerates all $128$ automorphisms and confirms that the
listed classes are full orbits.

For each representative the verifier constructs the quotient graph
and checks an explicit upper certificate on that graph. For outcomes
four and seven, a dual functional excludes degrees one and four,
respectively. Lemma~\ref{lem:congruence} makes these bounds exact.
Degree zero is impossible by Lemma~\ref{lem:quotient}. The six
representatives with degree seven are
\[
 (5,18),\ (5,21),\ (5,23),\ (9,21),\ (9,23),\ (11,23).
\]
Their orbits are exactly the $36$ pairs in $L\times R$. These are
orbits of pairs in the input graph; no count of isomorphism types of
the output graphs is asserted.

Each degree-seven quotient has a $189$-term upper certificate obtained
by substitution and a degree-four dual checked on all $814{,}875$
permitted monomial-edge multiples. Every residue class is included.
Thus $4{,}889{,}250$ checks establish the six degree-seven lower
bounds. The $64$ degree-four representatives require a further
$65{,}050$ dual checks, for $4{,}954{,}300$ in the classification.

For each degree-seven representative, the verifier checks proper
three-colorings after all $41$ edge deletions. There are no isolated
vertices, so deleting any vertex leaves a subgraph of an edge-deleted
graph. Recoloring one endpoint of a deleted edge with a fourth color
also gives a four-coloring of the full graph. Together with the upper
refutation, these facts prove $4$-criticality.

For the treewidth upper bound, identify the labels in a supplied
width-three decomposition of $T$ and add the merged vertex to any
bags needed for connectedness. The resulting width-four decomposition
is checked directly. For the lower bound, each representative has a
supplied minor with seven vertices, fourteen edges, and minimum
degree four. The verifier checks disjoint connected branch sets and
every minor edge. Every graph of treewidth at most three, and every
minor of such a graph, has a vertex of degree at most three. These
minors therefore prove treewidth at least four. All witnesses and
the verifier are included in \texttt{anc/controlled/}.
\end{proof}

In particular, let $Q=T/(5=18)$. Retain the smaller identified label,
remove label $18$, and renumber subsequent labels consecutively. Then
\[
 (|V(T)|,N(T),\tw(T))=(25,7,3),\qquad
 (|V(Q)|,N(Q),\tw(Q))=(24,7,4).
\]
Here $\tw(T)=3$ follows from its supplied decomposition and its
non-$3$-colorability, since graphs of treewidth at most two are
$3$-colorable. Treewidth is relevant because
Theorem~\ref{thm:separator} forces logarithmic degree whenever it is
bounded. The identification takes $Q$ outside the pure $K_4$-join
class while preserving degree. A uniform bound of four would still
give logarithmic degree; this single increase does not establish
the unbounded width required for linear degree.

\subsection{A compressed self-join without degree loss}

Write $\Gamma_G(v)$ for the neighbors of $v$. Adjacent vertices $a,b$
are \emph{true twins} if
$\Gamma_G(a)\setminus\{b\}=\Gamma_G(b)\setminus\{a\}$.
Form the \HJ\ self-join of $G$ along $ab$, identifying the two
copies of $a$. Choose $U\subseteq V(G)$ with $b\in U$ and $a\notin U$.
For every $v\notin U$ other than $a$, identify its left and right
copies; retain both copies of each vertex in $U$. Denote the result
by $C_U(G;a,b)$. All these identifications are legal: the only added
edge between the copies joins the two copies of $b$, which are not
identified.

\begin{proposition}\label{prop:controlled-join}
Let $k\geq3$, let the characteristic of $\F$ not divide $k$, and let
$G$ be non-$k$-colorable with a true-twin edge $ab$. Then
\begin{equation}\label{eq:controlled-bound}
 |V(C_U)|=|V(G)|+|U|,\qquad
 N(G)\leq N(C_U)\leq N(G)+k.
\end{equation}
Consequently $N(C_U)\in\{N(G),N(G)+k\}$. Criticality of $C_U$
is not assumed or implied.
\end{proposition}
\begin{proof}
The upper bound follows from Theorem~\ref{thm:main} and the
identifications, which do not increase degree. Define a graph
homomorphism $\rho:C_U\to G$ by fixing all left vertices, mapping
the right copy of $b$ to $a$, and mapping every other right vertex
$v$ to $v$. This is consistent with the identifications. A retained
right edge incident with $b$ maps to an edge incident with $a$ by
the true-twin condition. Every other retained edge maps to the same
edge of $G$, and the bridge maps to $ab$. Thus substitution
$x_v\mapsto x_{\rho(v)}$ sends a certificate for $C_U$ to one for
$G$ without increasing degree. This proves the lower bound.
The vertex count is immediate, and Lemma~\ref{lem:congruence}
leaves only the two displayed degree values.
\end{proof}

If all common external neighbors of $a,b$ are identified, the degree
cannot rise. In that case $G$ embeds in $C_U$ by mapping $a$ to the
right copy of $b$ and fixing every other vertex. The reverse
certificate bound then gives $N(C_U)=N(G)$. A possible increase
therefore requires at least one common neighbor to remain duplicated.

\subsection{A repeatable rule and its first plateau}

Call a true-twin edge eligible if its endpoints have degree at least
three. The graph $Q$ has the eligible edges
$(9,10),(11,12),(20,21),(22,23)$. Starting with $G_0=Q$, choose
an eligible edge $ab$, a disjoint eligible spare edge $uv$, and a
common external neighbor $w$ of $a,b$. Set
\begin{equation}\label{eq:four-rule}
 U=\{b,w,u,v\},\qquad G_{t+1}=C_U(G_t;a,b).
\end{equation}
Choose the lexicographically first eligible edge, the first disjoint
eligible spare edge, and the least common neighbor. Retain all left
labels and append the four right-copy labels in increasing order
of their original labels. These conventions make the rule deterministic.

\begin{proposition}\label{prop:four-rule}
The rule~\eqref{eq:four-rule} is available at every stage and satisfies
\begin{equation}\label{eq:four-rule-bounds}
 |V(G_t)|=24+4t,\qquad
 N(G_{t+1})\in\{N(G_t),N(G_t)+3\},\qquad
 7\leq N(G_t)\leq7+3t.
\end{equation}
\end{proposition}
\begin{proof}
Proposition~\ref{prop:controlled-join} prevents degree loss, so every
input has degree at least seven and cannot contain $K_4$, whose
degree is one. Distinct eligible true-twin edges are disjoint:
if two shared a vertex, their three endpoints would be mutually
adjacent true twins, and eligibility would supply a common fourth
neighbor, forming $K_4$. Thus, whenever two eligible edges exist,
the lexicographically selected edge has a disjoint spare.
If $w$ belonged to the spare pair, the two true-twin
conditions would put all six edges of $K_4$ on $a,b,u,v$. Hence the
four vertices in $U$ are distinct. Both left and right copies of
$uv$ remain disjoint true-twin edges, with degree at least three:
their common neighbors are either retained in the same copy or
identified together, and neither copy loses an edge at the selected
edge $ab$. They provide eligible disjoint pairs for the next round.
Induction proves availability and the order formula; the degree
bounds follow from Proposition~\ref{prop:controlled-join}.
\end{proof}

Each round uses one self-join and $|V(G_t)|-5$ pair identifications.
Through $t$ rounds there are $t$ joins and $2t^2+17t$ identifications
when the previous graph is reused for both inputs. This describes a
polynomial-size construction DAG; unfolding reuse into an expression
tree is a different size measure. The graph rule is computable in
polynomial time, but no bounded maximum degree or linear edge count
is asserted.

The first output has $28$ vertices and $51$ edges. A checked
$3{,}557$-term degree-seven certificate on a $25$-vertex subgraph
lifts to the output. Its homomorphism to $Q$ supplies the matching
lower bound, so $N(G_1)=7$. The subgraph is proper and non-$3$-colorable;
thus $G_1$ is not $4$-critical. This rule provides size control and
prevents degree loss, but does not force a rise even in its first
round. Its asymptotic degree growth is unknown.

\subsection{A necessary identification depth}

For a construction from $K_4$ using only joins and identifications,
let $i$ be the maximum number of pair-identification nodes on a
leaf-to-root path. This definition refers to the chosen construction
and also applies to a construction DAG.

\begin{corollary}\label{cor:identification-depth}
Over $\F_2$ for three-coloring, such a construction of a graph $G$
on $n$ vertices satisfies
\[
 \tw(G)\leq3+i,\qquad
 N(G)\leq2(i+4)(\lfloor\log_2 n\rfloor+1).
\]
In particular, linear degree requires identification depth
$\Omega(n/\log n)$.
\end{corollary}
\begin{proof}
A single identification increases treewidth by at most one; see
\cite[Theorem~6]{GW}. For completeness, replace the two labels by
their common image in all bags and add that label along the path
connecting their former bag subtrees. At most one vertex is added
to a bag. Combine this fact with Proposition~\ref{prop:join-width}
by induction through the construction. The degree bound then follows
from Theorem~\ref{thm:separator}.
\end{proof}

Together with Corollary~\ref{cor:depth}, linear degree therefore
requires join depth $\Omega(n)$ and identification depth
$\Omega(n/\log n)$. The two maximum depths need not occur on the
same path. Neither depth bound alone implies an exponential lower
bound on the number of construction nodes.

\section{Open problems and scope}
\label{sec:outlook}

Theorem~\ref{thm:seven} realizes an increase from four to seven, but
does not determine which selected edges force an increase. A first
remaining finite target is an example with a verified increase from
seven to ten. More generally, one seeks a rule forcing sufficiently
frequent increases while controlling the order of the output.

Corollary~\ref{cor:joins-only} rules out linear degree on linear
order using joins alone. The next target is an explicit family of
$4$-critical graphs constructed by specified joins and nonadjacent
vertex identifications with
\[
 |V(G_t)|=O(t),\qquad N(G_t)=\Omega(t).
\]
This is the optimal order of degree growth: every reduced monomial
has degree at most $(k-1)|V(G)|$, so $N(G)\leq(k-1)|V(G)|$.
Theorem~\ref{thm:separator} also makes treewidth
$\Omega(n/\log n)$ necessary for linear degree. Merely increasing
treewidth from three to four is insufficient.

Proposition~\ref{prop:four-rule} supplies linear order and prevents
degree loss, but gives no frequency of strict increases. Its first
output has degree seven and is not critical. Maintaining criticality,
bounded maximum degree, or linear edge count would each require a
further argument. Linear-degree constructions are already known
by~\cite{LN}; the target is a construction controlled by these graph
operations and a proof explaining its degree growth.
No claim of minimum graph order or certificate support is made.

\paragraph{Relation to $\mathrm P$ versus $\mathrm{NP}$.}
The family $T_{r,s}$ has certificates with a linear number of terms and
fixed degree seven. For fixed $d$, there are at most $\binom{n+d}{d}$
candidate multiplier monomials, so the degree-$d$ certificate problem
over $\F_2$ is a linear system of polynomial size. The finite lower
bounds exclude certificates of degree at most six; they give no
superpolynomial running-time bound. A Nullstellensatz degree lower
bound does not automatically give a degree lower bound for the stronger
polynomial calculus, nor does the Boolean degree-size tradeoff apply
without justification to roots-of-unity systems; see~\cite[Section~2.1]{LN}.
A lower bound for a restricted proof system also leaves other algorithms
unrestricted. These results do not settle $\mathrm P$ versus $\mathrm{NP}$.

\section{Computational verification}
\label{sec:verification}

The general bounds have algebraic proofs. Computation supplies the
finite certificates and checks the composition formulas. The ancillary
files contain the graph data, certificates, and Python programs using
exact arithmetic over $\F_2$. The sparse-polynomial checker is separate
from the search and composition code. It checks dual obstructions
against every reduced multiplier monomial within the stated degree bound.

The recorded checks comprise: 52 evaluations of the general local
identity for $2\leq k\leq12$ over permitted prime fields; explicit
certificates for the examples above; five expansions into ordinary
polynomial identities, including the vertex multipliers; and all 36
pairs of oriented-edge orbits in Moser--Moser joins. Each of the latter
has a degree-$4$ certificate, and its substitution through $\F_4$ and
coefficient descent to $\F_2$ was checked. These tests found no
degree-$7$ Moser--Moser join. The degree-seven construction in
Section~\ref{sec:degree7} instead joins two $13$-vertex chains.

As implementation controls, all 1,024 labeled graphs on five vertices
were tested at degree bounds $1$ and $4$. The resulting 132 refutations
and 1,916 dual obstructions were separately verified, with the search
outcomes agreeing with exhaustive three-coloring on this finite set.
The checker also rejected five deliberately corrupted certificates.

For $T_{4,4}$, two arithmetic implementations check the $189$-term upper
certificate and the dual functional described in
Section~\ref{sec:seven-lower}. Each checks all $947{,}141$ multiples
through degree four. The verifier also checks a three-coloring after
each of the $41$ edge deletions, and tests the construction and folding
formulas on five family members. The proof for arbitrary $r,s$ is given
in Section~\ref{sec:folding}.

For Theorem~\ref{thm:identification}, the standalone auditor checks
all $259$ pairs through $78$ full automorphism orbits, the upper
certificates, and $4{,}954{,}300$ dual constraints. It also checks
$246$ edge-deletion colorings and matching treewidth witnesses for
the six degree-seven representatives. Its arithmetic uses base-three
integer codes independently of the search implementation.

A separate auditor reconstructs ten rounds of
Proposition~\ref{prop:four-rule}, checks the deterministic choices,
the surviving spare twin pairs, and all $1{,}314$ edge images in
the homomorphisms. It checks the first-round degree-seven certificate
and, in a separate finite probe, degree-at-most-four certificates for
$17$ specified second identifications of $Q$. This probe is not an
exhaustive classification of quotients of $Q$. A third auditor checks
$11$ width-three decompositions and two expanded separator certificates.
The asymptotic statements in Section~\ref{sec:joins-only} rest on
their written proofs, not on these finite tests.

\paragraph{Acknowledgment.}
The author acknowledges the use of OpenAI GPT-6 through Codex as a tool
in the research, computational work, and preparation of this manuscript.
Responsibility for the results and conclusions rests with the author.

\clearpage
\appendix
\section{Explicit finite witnesses}
\label{app:witnesses}

All calculations in this appendix are over $\F_2$, modulo $x_i^3=1$.
Write $q_{ij}=x_i^2+x_ix_j+x_j^2$. Linear monomials are written as $x_i$.
For higher degrees, the word $\mathtt{i_1\cdots i_t}$ denotes
$x_{i_1}\cdots x_{i_t}$; for example, $\mathtt{0018}=x_0^2x_1x_8$.
Every vertex label here is a single digit. Addition is in $\F_2$.

\begingroup
\small
\noindent\begin{minipage}{\linewidth}
\subsection{A degree-one certificate for \texorpdfstring{$K_4$}{K4}}
The following multipliers satisfy $\sum_{ij\in E(K_4)}h_{ij}q_{ij}=1$.
Unlisted edge multipliers are zero. Direct ordinary expansion gives
$\sum h_{ij}q_{ij}=x_1^3$, which is $1$ in the quotient.
\begin{center}
\small
\begin{tabular}{@{}cl@{}}
\toprule
$ij$ & $h_{ij}$ \\
\midrule
$01$ & $x_{0}$ \\
$02$ & $x_{0}$ \\
$03$ & $x_{1}+x_{2}$ \\
$12$ & $x_{1}$ \\
$13$ & $x_{0}+x_{2}$ \\
$23$ & $x_{0}+x_{1}$ \\
\bottomrule
\end{tabular}
\end{center}
\end{minipage}\par

\subsection{A degree-one obstruction for \texorpdfstring{$M$}{M}}
Define an $\F_2$-linear functional $\Lambda$ on reduced monomials of
degree at most three. Set $\Lambda(1)=1$ and set $\Lambda$ to one
on exactly the following cubic monomials, and to zero on all remaining
monomials in this degree range:
\[\begin{gathered}
\mathtt{003},\;\mathtt{006},\;\mathtt{012},\;\mathtt{022},\;\mathtt{023},\;\mathtt{045},\;\mathtt{055},\;\mathtt{056} \\
\mathtt{113},\;\mathtt{115},\;\mathtt{122},\;\mathtt{123},\;\mathtt{125},\;\mathtt{126},\;\mathtt{135},\;\mathtt{136} \\
\mathtt{144},\;\mathtt{145},\;\mathtt{146},\;\mathtt{226},\;\mathtt{233},\;\mathtt{336},\;\mathtt{344},\;\mathtt{355} \\
\mathtt{366},\;\mathtt{446},\;\mathtt{455},\;\mathtt{456},\;\mathtt{566}
\end{gathered}\]
For every $ij\in E(M)$, direct substitution in this list gives
$\Lambda(q_{ij})=0$ and $\Lambda(x_tq_{ij})=0$ for
$t=0,\ldots,6$, reducing exponents modulo three before applying the
functional. These are $88$ checks: $11$ constant multipliers and
$77$ linear multipliers. Hence a degree-one identity would imply
$1=\Lambda(1)=0$. A degree-four certificate for $M$ is obtained
directly by applying~\eqref{eq:composition} to the two copies of the
$K_4$ certificate above.

\noindent\begin{minipage}{\linewidth}
\subsection{A degree-four certificate for \texorpdfstring{$C$}{C}}
For the edge set in~\eqref{eq:C}, the following multipliers satisfy
$\sum h_{ij}q_{ij}=1$ after reduction modulo $x_i^3=1$.
Unlisted edge multipliers are zero. Each monomial has degree one or four.
\begin{center}
\small
\begin{tabular}{@{}cl@{}}
\toprule
$ij$ & $h_{ij}$ \\
\midrule
$03$ & $\mathtt{0018}+\mathtt{0019}+\mathtt{0028}+\mathtt{0029}$ \\
$05$ & $x_{0}+\mathtt{0011}+x_{1}$ \\
$06$ & $\mathtt{0014}+\mathtt{0015}+\mathtt{0114}+\mathtt{0115}+x_{4}+x_{5}$ \\
$08$ & $x_{0}+\mathtt{0019}+\mathtt{0119}+x_{2}+x_{7}$ \\
$09$ & $x_{0}+\mathtt{0018}+\mathtt{0118}+x_{2}+x_{7}$ \\
$12$ & $\mathtt{0018}+\mathtt{0019}$ \\
$13$ & $\mathtt{0028}+\mathtt{0029}+x_{8}+x_{9}$ \\
$14$ & $x_{0}+\mathtt{0011}+x_{1}$ \\
$23$ & $\mathtt{0018}+\mathtt{0019}+x_{8}+x_{9}$ \\
$27$ & $x_{8}+x_{9}$ \\
$45$ & $\mathtt{0014}+\mathtt{0114}+x_{4}$ \\
$46$ & $x_{0}+\mathtt{0011}+\mathtt{0015}+\mathtt{0115}+x_{1}+x_{5}$ \\
$56$ & $x_{0}+\mathtt{0011}+\mathtt{0014}+\mathtt{0114}+x_{1}+x_{4}$ \\
$78$ & $x_{0}+x_{2}+x_{7}$ \\
$79$ & $x_{0}+x_{2}+x_{7}$ \\
$89$ & $\mathtt{0018}+\mathtt{0118}$ \\
\bottomrule
\end{tabular}
\end{center}
\end{minipage}\par

\endgroup

\section{Reproduction}
\label{app:reproduction}

With Python 3.10 or newer, run the original examples from \texttt{anc/}:
\begin{verbatim}
python hajos_experiments.py
\end{verbatim}
The script regenerates and checks the witnesses, then writes
\texttt{hajos\_results.json}. In the JSON format, graph edges are listed
lexicographically. Each refutation term pairs an edge index with a list
of variable indices representing a monomial; terms are summed modulo two.
A dual obstruction lists the monomials on which its functional is one
and assigns zero to all others.

\subsection{Degree-seven artifacts}

The self-contained directory \texttt{anc/degree7/} contains the base graph,
upper and lower certificates, two exact checkers, the family construction,
and a hash manifest. With Python 3.10 or newer, run from that directory:
\begin{verbatim}
python -X utf8 verify_all.py
\end{verbatim}
The verifier uses only the Python standard library. It checks the base
graph and certificates, the diamond identity, the input degree, the
edge-deletion colorings, five family instances, and the file hashes.
Any failed check raises an error.

The central certificate artifacts have the following SHA-256 digests.
Concatenate the two lines for each file; the complete manifest is
\texttt{anc/degree7/SHA256SUMS.json}.
\begin{center}\small
\begin{tabular}{@{}ll@{}}
\toprule
File & SHA-256 \\
\midrule
\texttt{base\_upper.json} & \texttt{8fe1561eb7f1718e925948bd2d13274b} \\
 & \texttt{4b41b6283a2d7cf347ea30101c6e1bc1} \\
\texttt{base\_lower.json} & \texttt{95d3f0667521f53af9cd80dde8696b98} \\
 & \texttt{9d60f46ec21e6e35f2a21bb36e9ad7d0} \\
\bottomrule
\end{tabular}
\end{center}

The lower functional is specified by its nonzero support; all unlisted
monomials have value zero. The archive also contains an earlier
$283$-term upper certificate in \texttt{base\_upper\_original.json}.

\subsection{Identifications and structural bounds}

The directory \texttt{anc/controlled/} is a self-contained snapshot
of the three additional auditors and all their required graph,
certificate, coloring, and decomposition data. Run from that directory:
\begin{verbatim}
python -X utf8 audit_join_barrier.py
python -X utf8 audit_identifications.py
python -X utf8 audit_controlled_join.py
\end{verbatim}
Only the Python standard library is required. The auditors reconstruct
the named base graph and quotients, verify exact identities and duals,
and check the combinatorial witnesses. The compressed construction
is specified by~\eqref{eq:four-rule} and the labeling conventions
in Section~\ref{sec:identifications}; the ten stored rounds are
checked against that rule. \texttt{SOURCE\_PROVENANCE.json} records
the source hashes. The archive-wide manifest is \texttt{anc/SHA256SUMS}.


\begin{thebibliography}{9}

\bibitem{DLMRS}
J.~A. De Loera, S.~Margulies, M.~Pernpeintner, E.~Riedl, D.~Rolnick,
G.~Spencer, D.~Stasi, and J.~Swenson,
\emph{Gr\"obner Bases and Nullstellens\"atze for Graph-Coloring Ideals},
arXiv:1410.6806, 2014.
\url{https://arxiv.org/abs/1410.6806}.

\bibitem{LLO}
B.~Li, B.~Lowenstein, and M.~Omar,
\emph{Low Degree Nullstellensatz Certificates for 3-Colorability},
Electronic Journal of Combinatorics \textbf{23}(1) (2016), Paper P1.6.
Preprint: \url{https://arxiv.org/abs/1503.04680}.

\bibitem{RT}
J.~Romero and L.~Tun\c{c}el,
\emph{Graphs with Large Girth and Chromatic Number are Hard for
Nullstellensatz}, manuscript revised February 12, 2024.
\url{https://www.math.uwaterloo.ca/~ltuncel/publications/2212.05365.pdf}.
Earlier version: arXiv:2212.05365, 2022.

\bibitem{LN}
M.~Lauria and J.~Nordstr\"om,
\emph{Graph Colouring is Hard for Algorithms Based on Hilbert's
Nullstellensatz and Gr\"obner Bases},
32nd Computational Complexity Conference (CCC 2017),
LIPIcs \textbf{79} (2017), 2:1--2:20.
\url{https://doi.org/10.4230/LIPIcs.CCC.2017.2}.
Full version: \url{https://arxiv.org/abs/2306.00125}.

\bibitem{GW}
F.~Gurski and R.~Weishaupt,
\emph{The Behavior of Tree-Width and Path-Width Under Graph Operations
and Graph Transformations},
Algorithms \textbf{18}(7) (2025), 386.
\url{https://doi.org/10.3390/a18070386}.

\end{thebibliography}
\end{document}